\documentclass[11pt]{article}
\usepackage{amsmath, amssymb, theorem, latexsym, epsfig}
\numberwithin{equation}{section}

\theoremstyle{plain}
\theorembodyfont{\itshape}
\newtheorem{theorem}{Theorem}[section]
\newtheorem{proposition}[theorem]{Proposition}
\newtheorem{lemma}[theorem]{Lemma}
\newtheorem{corollary}[theorem]{Corollary}
                                                                                
\theorembodyfont{\rmfamily}
\newtheorem{definition}[theorem]{Definition}

\newtheorem{example}[theorem]{Example}
\newtheorem{remark}[theorem]{Remark}
\newtheorem{convention}[theorem]{Convention}
\newenvironment{proof}{{\noindent \textbf{Proof}\,\,}}{\hspace*{\fill}$\Box$\medskip}

\title{On odd-periodic orbits in complex planar billiards}
\author{Alexey Glutsyuk
\thanks{Permanent address: CNRS, Unit\'e de Math\'ematiques
Pures et Appliqu\'ees, M.R., \'Ecole Normale Sup\'erieure de Lyon,
46 all\'ee d'Italie, 69364 Lyon 07, France.  \newline Email:
aglutsyu@ens-lyon.fr}
\thanks{Laboratoire
J.-V.Poncelet (UMI 2615 du CNRS et l'Universit\'e Ind\'ependante
de Moscou)}
\thanks{National Research University Higher School of Economics, Russia}
 \thanks{Supported by part by RFBR grants  
10-01-00739-a,  13-01-00969-a, 
 NTsNIL\_a (RFBR-CNRS) grant  10-01-93115 and by ANR grant ANR-08-JCJC-0130-01.}}
\begin{document}
\maketitle
\begin{abstract} The famous conjecture of V.Ya.Ivrii (1978) says that 
{\it in every billiard with infinitely-smooth boundary in a Euclidean space 
the set of periodic orbits has measure zero}. In the present paper we study the 
complex version of Ivrii's conjecture for odd-periodic orbits in planar billiards, with reflections 
from complex analytic curves. We prove  positive 
answer in the following cases: 1) triangular orbits; 2)  odd-periodic orbits in the case, when the mirrors are 
algebraic curves avoiding two special points at infinity, the so-called isotropic points. We provide  immediate applications to 
the real piecewise-algebraic Ivrii's conjecture and to its analogue in the invisibility theory. 
\end{abstract}

\def\cc{\mathbb C}
\def\oc{\overline{\cc}}
\def\cp{\mathbb{CP}}
\def\wt#1{\widetilde#1}
\def\rr{\mathbb R}
\def\var{\varepsilon}
\def\tt{\mathcal T}
\def\mcr{\mathcal R}
\def\a{\alpha}

\section{Introduction}

The famous V.Ya.Ivrii's conjecture \cite{Ivrii} 
 says that {\it in every billiard with infinitely-smooth 
boundary in a Euclidean space of any dimension the set of periodic orbits 
has measure zero.} As it was shown by V.Ya.Ivrii \cite{Ivrii}, it implies the famous 
H.Weyl's conjecture  on the two-term 
asymptotics of the spectrum of Laplacian \cite{HWeyl11}. A brief historical survey of 
both conjectures with references is presented in \cite{gk2}. For triangular orbits Ivrii's conjecture 
was proved in  \cite{bzh, rychlik, stojanov, vorobets, W}. For quadrilateral orbits it was proved in 
\cite{gk1, gk2}. 

\begin{remark} \label{rem-iv} 
Ivrii's conjecture is open 
already for piecewise-analytic billiards, and we believe that this is its principal case. In the latter case 
 Ivrii's conjecture is equivalent to the statement saying that for every 
$k\in\mathbb N$ the set of $k$-periodic orbits has empty interior. In the case, when the boundary is analytic, regular and convex, 
this was proved for arbitrary period in \cite{vas}. 
\end{remark}

 In the present paper we study a complexified version of Ivrii's  
conjecture in complex dimension two for odd periods. More precisely, we consider the complex plane 
$\cc^2$ equipped with the complexified Euclidean metric, which is the standard complex-bilinear 
quadratic form. This defines notion of symmetry with respect to a complex line. 
Reflections of complex lines with respect to complex analytic  curves are defined by 
the same formula, as in the real case. See \cite[subsection 2.1]{alg}  and Subsection 2.2 below for more detail. 

\begin{remark}
Ivrii's conjecture has an analogue in the invisibility theory, see Subsection 1.2 and references therein. It appears that both conjectures have 
the same complexification. Thus, results on the complexified Ivrii's conjecture have applications to both 
Ivrii's conjecture and invisibility. 
\end{remark}

Main results and an application to the real Ivrii's conjecture  are stated in Subsection 1.1. Corollary on the invisibility is stated and proved in Subsection 1.2. 

\subsection{Complex billiards, main results and plan of the paper.} 

\begin{definition} A complex projective line $l\subset\cp^2\supset\cc^2$ is {\it isotropic}, 
if either it  coincides with the infinity line, or the complexified Euclidean quadratic 
form on $\cc^2$ vanishes on $l$. Or equivalently, a line is isotropic, if it passes through some 
of two points with homogeneous coordinates $(1:\pm i:0)$: the  {\it isotropic 
points at infinity}.  In what follows we denote the latter  points by 
$$I_1=(1:i:0), \ \ I_2=(1:-i:0).$$
\end{definition}

\begin{definition} \cite{alg} A planar {\it complex analytic (algebraic) billiard} is a finite collection 
of complex analytic (algebraic) 
curves-``mirrors'' $a_1,\dots,a_k$. We assume that no mirror $a_j$ is an isotropic line and set $a_0=a_k$, $a_{k+1}=a_1$.  
\end{definition}


\begin{definition} \label{deforb} \cite{alg} 
A {\it $k$-periodic billiard orbit} is a collection of points $A_j\in a_j$, $A_{k+1}=A_1$, $A_k=A_0$, 
such that for every $j=1,\dots,k$ one has 
$A_j\neq A_{j+1}$, the tangent line $T_{A_j}a_j$ is not isotropic and 
the  complex lines $A_{j-1}A_j$ and $A_jA_{j+1}$ 
are transverse to it  and  symmetric  with respect to it. (Properly saying, we have to take points $A_j$ together with 
prescribed branches of curves $a_j$ at $A_j$: this specifies the  line $T_{A_j}a_j$ in unique way, if $A_j$ is a self-intersection point 
of the curve $a_j$.) 
 \end{definition}

\begin{remark} In a real billiard the reflection of a ray from the boundary 
is uniquely defined: the reflection is made at the first point where the ray meets the 
boundary. In the complex case,  the reflection of lines with respect to a 
complex analytic curve is a multivalued mapping 
{\it (correspondence)} of the space of lines in $\cp^2$: we do not have a canonical choice of intersection point of a line with the curve. Moreover, the notion of interior domain 
does not exist in the complex case, since the mirrors have real codimension two. 
\end{remark}

\begin{definition} \cite{alg} A complex analytic billiard $a_1,\dots,a_k$ is {\it $k$-reflective,} if 
it has an open set of periodic orbits. In more detail this means that there exists 
an open set of pairs $(A_1,A_2)\in a_1\times a_2$ extendable to $k$-periodic 
orbits $A_1\dots A_k$. (Then the latter property automatically holds for every other 
pair of neighbor mirrors $a_j$, $a_{j+1}$.) 
\end{definition}

{\bf Problem (Complexified  version of Ivrii's conjecture) \cite{alg}.}  
{\it Classify all the  $k$-reflective complex analytic (algebraic)  billiards.}
\medskip

It is known that there exist 4-reflective complex planar algebraic billiards, see  \cite[p.59, corollary 4.6]{tab}  and \cite{alg}. 
Their complete classification is given  in \cite{alg}. This implies existence of $k$-reflective algebraic billiards for all $k\equiv0(mod 4)$, 
see \cite[remark 1.5]{alg}. 

\medskip

{\bf Conjecture.}  There are no $k$-reflective complex analytic (algebraic) planar billiards for odd $k$.

\medskip

The next two theorems partially confirm this conjecture.

\begin{theorem} \label{three} Every  planar complex analytic billiard with three mirrors 
is not 3-reflective. 
\end{theorem}

\begin{theorem} \label{odd} Let a planar complex algebraic billiard have odd number $k$ of mirrors, and let 
each mirror contain no isotropic point at infinity. Then the billiard is not $k$-reflective.
\end{theorem}

Theorem \ref{three} is the complexification of the above-mentioned  results by M.Rychlik et al  on triangular orbits in real billiards, 
see \cite{bzh, rychlik, stojanov, vorobets, W}. 
Theorem \ref{odd} has immediate application to the real Ivrii's conjecture.  

\begin{corollary} \label{cor-iv} Consider a real planar  billiard with piecewise-algebraic boundary. Let the complexifications of its 
algebraic pieces contain no isotropic point at infinity.  Then the set of its odd-periodic orbits has measure zero.
\end{corollary}

The corollary follows immediately from Theorem \ref{odd} and Remark \ref{rem-iv}. 

Theorem \ref{odd} is proved in Section 3. Theorem \ref{three} is proved in Section 4. 
Their proofs are based on the  following elementary fact.

\begin{proposition} \label{is-refl} The symmetry with respect to a non-isotropic line permutes the isotropic directions: the 
image of an  isotropic line through the isotropic point $I_1$ at infinity 
 passes through the other isotropic point $I_2$.
\end{proposition}

Proposition \ref{is-refl} follows from a proposition at the beginning of  \cite[subsection 2.1]{alg}. 

\begin{corollary} \label{is-even} 
Let a periodic orbit in complex planar analytic billiard have finite vertices, 
and at least one of its edges  (complex lines through neighbor vertices) be isotropic.  
Then all the edges are isotropic, and their 
directions (corresponding isotropic points at infinity) are intermittent, see Fig.1. In particular, the period is even.
\end{corollary}

\begin{figure}[ht]
  \begin{center}
   \epsfig{file=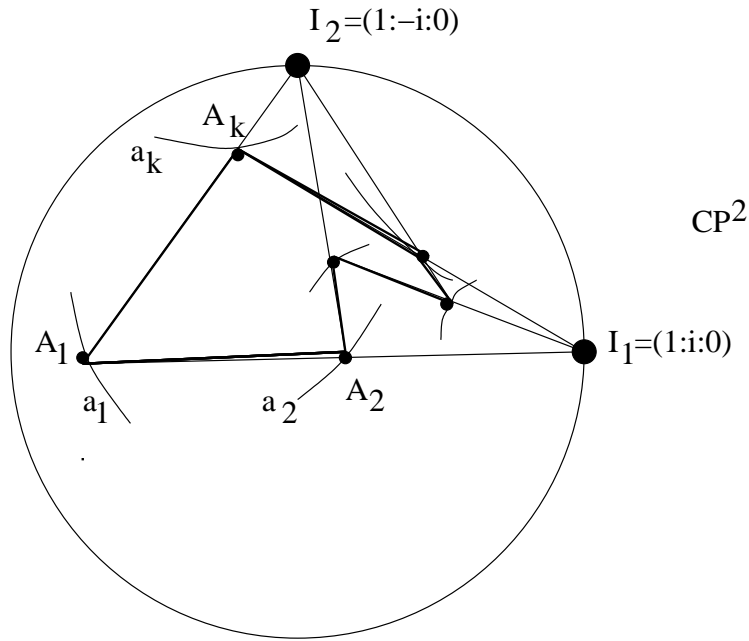}
    \caption{A periodic  orbit with isotropic edges of intermittent directions}
    \label{fig:1}
  \end{center}
\end{figure}

We prove Theorem \ref{odd} by contradiction. Supposing the contrary, i.e., the existence of 
an open set of odd-periodic orbits,  we show that it contains a finite orbit with an isotropic edge, as in the latter corollary. 
This is the main technical part of the proof, and this is the place we use the second technical assumption of Theorem \ref{odd}. 
This together with Corollary \ref{is-even} implies that the period should be even, -- a contradiction.

For  the proof of Theorem \ref{three}, supposing the contrary, 
we prove the existence of a one-dimensional family of orbits with one isotropic edge through two  
variable vertices so that the third vertex is a fixed isotropic point at infinity. 
We show that the existence of the latter family contradicts the reflection law at the 
third vertex. In the proof we deal with the maximal analytic extensions of mirrors and 
the closure of the open set of periodic orbits in the product of the extended mirrors. The corresponding background material 
and basic facts about complex reflection law are contained 
in  Subsections 2.1 and 2.2 respectively and  in  \cite[subsection 2.1]{alg}. 

\subsection{Corollaries for  the invisibility}
\def\ha{\hat a}
\def\hb{\hat b}
\def\hc{\hat c}

This subsection is devoted to Plakhov's Invisibility Conjecture: the analogue of Ivrii's conjecture in the invisibility theory 
\cite[conjecture 8.2]{pl}. We recall it below and show that it follows from a conjecture saying that no finite collection of germs of 
smooth curves can form a $k$-reflective billiard for appropriate ``invisibility'' reflection  law. In the case, when the curves are analytic, 
the invisibility reflection law is a real form of complex reflection law. This shows that both invisibility and Ivrii's conjectures have the same 
complexification. For simplicity we present this relation in dimension two.  We state and prove 
Corollaries \ref{invis1} and \ref{invis2} of our complex results (Theorems \ref{three} and \ref{odd}) for planar Invisibility Conjecture.

\begin{definition}
Consider an arbitrary perfectly reflecting (may be disconnected) closed bounded body $B$ in a Euclidean space.  For every oriented line $R$ 
take its first intersection point $A_1$ with the boundary $\partial B$ and reflect $R$ from the tangent hyperplane $T_{A_1}\partial B$. The 
reflected ray goes from the point $A_1$ and defines a new oriented line. Then we repeat this procedure. Let us assume that after a finite 
number of reflections the output oriented line coincides with the input line $R$ and will not hit the body any more. 
Then we say that the body $B$ {\it is invisible in the direction 
$R$,} see Fig.2. We call $R$ the {\it invisibility direction,} and the finite piecewise-linear curve  bounded by the first 
and last reflection points will be called its  {\it complete trajectory}.  
 \end{definition}

 \begin{figure}[ht]
  \begin{center}
   \epsfig{file=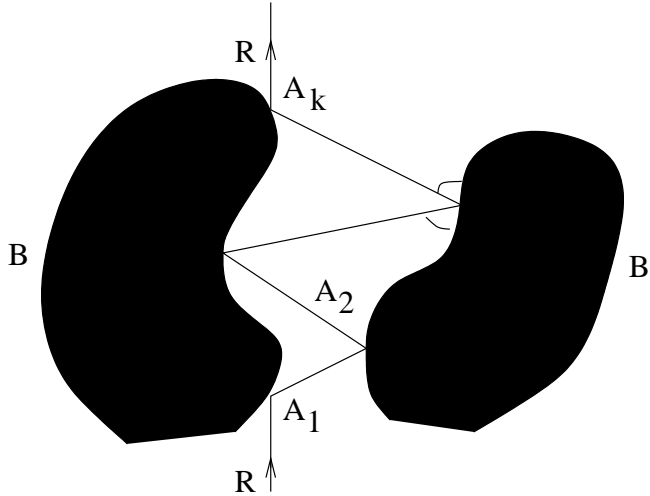}
    \caption{A body invisible in one direction.}
    \label{fig:2}
  \end{center}
\end{figure}

{\bf Invisibility Conjecture} (A.Plakhov, \cite[conjecture 8.2, p.274]{pl}.) 
{\it There is no body with piecewise $C^{\infty}$ boundary for which the set of invisibility directions has positive measure.}

\begin{remark} As is shown by A.Plakhov in his book \cite[section 8]{pl}, there exist no body invisible in all directions. The same book 
contains a very nice survey on invisibility, including examples of bodies invisible in a finite number of (one-dimensional families of) 
directions. See also papers \cite{ply,pl2,pl3} for more results. 
The Invisibility Conjecture is open even in dimension 2.  It is equivalent to the statement saying that there are no 
$k$-reflective bodies for every $k$, see the next definition. 
\end{remark}

\begin{definition} A body $B$  with piecewise-smooth boundary is called {\it $k$-reflective,} if the set of invisibility directions 
with $k$ reflections has positive measure. 
\end{definition}

\begin{definition} Let  $a_1,\dots,a_k$ be a collection of (germs of) planar smooth curves. A $k$-gon $A_1\dots A_k$ with 
$A_j\in a_j$, $A_{k+1}=A_1$, $A_{0}=A_k$ is said to be a {\it  $k$-invisible orbit},   if 

- $A_j\neq A_{j+1}$ for every $j=1,\dots,k$;

-  the tangent line $T_{A_j}a_j$ is the exterior bisector of the angle $\angle A_{j-1}A_jA_{j+1}$ whenever $j\neq1,k$, and it is 
its interior bisector for $j=1,k$, see Fig.3. 

We say that the collection $a_1,\dots,a_k$ is a {\it k-invisible billiard}, if the set of its $k$-invisible orbits has positive measure. 
\end{definition}

 \begin{figure}[ht]
  \begin{center}
   \epsfig{file=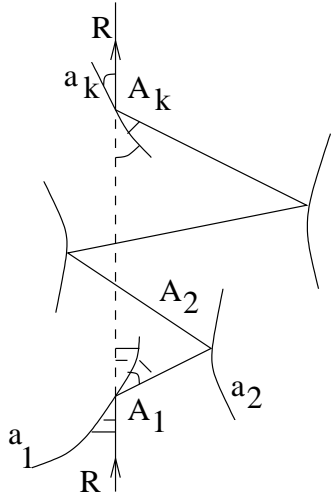}
    \caption{A $k$-invisible $k$-gon: new, invisibility reflection law at $A_1$ and $A_k$.}
    \label{fig:3}
  \end{center}
\end{figure}

\begin{proposition} \label{inv-ivr} Let $k\in\mathbb N$ and $B\subset\rr^2$ be a body 
 such that  no collection of $k$ germs of its boundary 
 forms a $k$-invisible billiard. Then the body $B$ is not $k$-reflective. 
 \end{proposition}

Proposition \ref{inv-ivr} is implicitly contained in \cite[section 8]{pl}. 

\begin{proposition} \label{invcomp} 
Let a collection of $k$ germs of planar analytic curves be a $k$-invisible billiard. Then its complexification is a 
$k$-reflective billiard. 
\end{proposition}

The proposition follows from definition and analyticity: both the usual reflection law and the invisibility reflection law at $A_1$ and $A_k$ 
from the above definition are two different real forms of the complex reflection  law.

\begin{corollary} \label{invis1} There are no 3-reflective bodies in $\rr^2$ with piecewise-analytic boundary.
\end{corollary}

\begin{remark} 
Corollary \ref{invis1} is known to specialists. 
As it is stated  in A.Plakhov's book \cite{pl} (after conjecture 8.2),  Corollary \ref{invis1} can be proved  by adapting  
the proof of Ivrii's conjecture for triangular orbits.  A.Plakhov's unpublished proof of Corollary \ref{invis1}  follows \cite{W}. 
\end{remark}

\begin{corollary} \label{invis2} Let $B\subset\rr^2$ be a body with piecewise-algebraic boundary, and let the complexifications of its algebraic pieces contain no isotropic point at infinity. Then $B$ is not $k$-reflective for every odd $k$.
\end{corollary}

Corollaries  \ref{invis1} and \ref{invis2} follow from Propositions \ref{inv-ivr}, \ref{invcomp} and 
Theorems \ref{three} and \ref{odd}, analogously to Corollary \ref{cor-iv}.

\section{Maximal analytic extension and complex reflection law}

\subsection{Maximal analytic extension}
Recall that a germ $(a,A)\subset\cp^n$ of analytic curve is {\it irreducible}, if it is the image of a germ of analytic mapping 
$(\cc,0)\to(a,A)$. 
\begin{definition} \label{order} Consider two holomorphic mappings of Riemann surfaces $S_1$, $S_2$ with base points $s_1\in S_1$ and 
$s_2\in S_2$ 
to $\cp^n$, $f_j:S_j\to\cp^n$, $j=1,2$, $f_1(s_1)=f_2(s_2)$. We say that 
$f_1\leq f_2$, if there exists a holomorphic 
mapping $h:S_1\to S_2$, $h(s_1)=s_2$, such that $f_1=f_2\circ h$. This defines a partial order on the set of classes of Riemann surface 
mappings to $\cp^n$ up to conformal reparametrization respecting base points. 
\end{definition}

\begin{proposition} \label{ext} Every irreducible germ of analytic curve in $\cp^n$ has maximal analytic extension. In more detail, let  
 $(a,A)\subset\cp^n$ be an irreducible germ of analytic curve. There exists an abstract Riemann surface $\hat a$ with base 
 point $\hat A\in\hat a$ (the so-called {\bf maximal normalization} of the germ $a$) 
 and a holomorphic mapping $\pi_a:\ha\to \cp^n$, $\pi_a(\hat A)=A$ with the following properties:
 
 -  the image of germ at $\hat A$ of the mapping $\pi_a$ 
 is contained in $a$; 
 
 -  $\pi_a$ is the maximal mapping with the above property in the sense of Definition \ref{order}.
 
 Moreover, the mapping $\pi_a$ is unique up to composition with conformal isomorphism of Riemann surfaces respecting base points. 
 \end{proposition} 

\begin{proof} The proposition is classical, and some specialists believe it goes up to Weierstrass. Let us give its proof for completeness of 
presentation. Let $\Psi$ denote the set of 
 all the piecewise-analytic paths $\gamma:[0,1]\to\cp^n$, $\gamma(0)=A$ with analytic pieces $\gamma([t_{j-1},t_j])$, $0=t_0<t_1<\dots
 <t_N=1$, that have the following properties:
 
 - the image of the germ at 0 of the mapping $\gamma$ lies in the germ $(a,A)$; 
 
 - if $\gamma\not\equiv const$, then $\gamma|_{[t_{j-1},t_j]}\not\equiv const$ for every $j=1,\dots,N$; 
 
 - for every $j$ the images of germs at $t_j$ of both mappings $\gamma|_{[t_{j-1},t_j]}$ and $\gamma|_{[t_j,t_{j+1}]}$ 
 lie in one and the same irreducible germ of analytic curve at $\gamma(t_j)$. 

Every path $\gamma\in\Psi$ is contained 
in a unique irreducible germ $\Gamma$ at $\gamma([0,1])$ of analytic curve. In particular, 
for every $\gamma\in\Psi$, set $g=\gamma(1)$, the germ of the path $\gamma$ at 1 is contained in 
a unique irreducible germ $\Gamma^1$ of analytic curve at $g$. 
We say that two paths $\gamma_1,\gamma_2\in\Psi$  are {\it equivalent}, if $g_1=g_2$ and $\Gamma_1^1=\Gamma_2^1$. 
 Let $\ha$ denote the set of all the equivalence classes of paths from $\Psi$. 
The $C^0$-topology on the space of paths $[0,1]\to\cp^n$ induces the quotient topology 
on the set $\ha$.  There is a natural projection $\pi_a:\ha\to\cp^n$: $\gamma\mapsto\gamma(1)$. 

\medskip

{\bf Claim.} {\it The set $\ha$ equipped with 
the induced topology admits a natural structure of Riemann surface so that the projection $\pi_a$ is holomorphic.}

\begin{proof} 
The space $\ha$ is identified with an appropriate set of irreducible germs of analytic curves in $\cp^n$. For every path 
$\gamma\in\Psi$ there exists an $\var>0$ such that every path in $\Psi$ $\var$-close to $\gamma$ lies in the analytic curve 
germ $\Gamma\supset\gamma([0,1])$. This follows from 
definition. Hence, each germ $(\Gamma^1,g)\in\ha$ admits a basis of neighborhoods that are identified with neighborhoods 
of the marked point $g$ in the local analytic curve $\Gamma^1$. In particular, the space $\ha$ is Hausdorff. 
Now for the proof of the claim it suffices to show that the space  $\ha$ has a countable basis: 
then the Riemann surface structure and holomorphicity of projection are immediate. Let us fix an affine chart  $\cc^n\subset\cp^n$ with 
the origin at $A$. Let $L\subset\cp^n$ be a coordinate line such that the coordinate projection of the germ $(a,A)$ to $L$ is non-constant. 
Fix  real coordinates $(x,y)$ on $L$. Let $\Lambda$ denote the set of paths $\gamma\in\Psi$  that are projected to 
``rational rectilinear paths'': piecewise-linear paths in $L$ with vertices having rational coordinates and with edges being 
parallel to $x$ and $y$ axes. The countable subset $\Lambda\subset\Psi$ is dense: each path $\gamma\in\Psi$ can be 
obviously approximated by liftings to $\Gamma$ of rational rectilinear paths. For every analytic curve in $\cp^n$ we measure 
distances between its points in the intrinsic metric induced by the Fubini--Studi metric of the projective space. For every 
$\gamma\in\Lambda$ let us consider the corresponding germ $(\Gamma^1,g)$ and take those $2^{-n}$-neighborhoods  in $\Gamma^1$ 
of the point $g$ that are relatively compact in the Riemann surface $\Gamma^1$. They  are canonically identified with 
neighborhoods of the point $[\gamma]\in\ha$. Now let us cover the projective space by a finite number of affine charts and 
construct similar neighborhoods with respect to each chart. The neighborhoods thus constructed 
form a countable basis of topology of the space $\ha$, which 
follows immediately from definition and construction. This proves the claim.
\end{proof}

Thus, the set $\ha$ is a Riemann surface, and the projection $\pi_a:\ha\to\cp^n$ is an analytic extension of the germ $a$. 
Let us show that this is a maximal analytic extension. Let $\phi:S\to\cp^n$ be a holomorphic mapping of a  Riemann surface $S$, 
and its germ at a base point $s\in S$ parametrizes the germ $(a,A)$ 
(not necessarily bijectively). Consider the mapping $h=\pi_a^{-1}\circ\phi$, which 
is holomorphic and well-defined in 
a neighborhood of the point $s$. It extends up to a holomorphic mapping $h:S\to\ha$ such that $\phi=\pi_a\circ h$. Indeed, it 
extends analytically along every locally-nonconstant piecewise-analytic 
path $\alpha:[0,1]\to S$ starting at $s$,  and one has $\phi\circ\alpha\in\Psi$, by construction. The result of analytic extension 
depends only on the end-point $\alpha(1)$, since  $\phi$ is holomorphic single-valued  
and by the definition of the space $\ha$. This proves the maximality of the mapping $\pi_a$. Let us prove that a maximal 
mapping is unique up to composition with conformal isomorphism. Indeed, let $\phi_1:S_1\to\cp^n$ and $\phi_2:S_2\to\cp^n$ 
be two maximal mappings, whose germs at $s_1\in S_1$ and $s_2\in S_2$ parametrize the germ $(a,A)$. It follows from 
maximality that both latter local parametrizations are 1-to-1. Therefore, there exists a unique germ $h:(S_1,s_1)\to (S_2,s_2)$ such that 
$\phi_1=\phi_2\circ h$. It should extend holomorphically to $S_1$, by maximality of the mapping $\phi_2$, and its inverse should 
extend to $S_2$, by maximality of the mapping $\phi_1$. Thus, $h:S_1\to S_2$ is a conformal isomorphism.  Proposition \ref{ext} 
is proved.
\end{proof}

\begin{example} The maximal normalization of a projective algebraic curve is its usual normalization: 
a compact Riemann surface parametrizing the curve bijectively, except for self-intersections. 
\end{example}

\subsection{Complex reflection law}

\def\mcl{\mathcal L}
The material presented in this subsection is contained in \cite[subsection 2.1]{alg}, except for Corollary \ref{ccoinc}.

We fix an Euclidean metric on $\rr^2$ and consider its complexification: the 
complex-bilinear quadratic form $dz_1^2+dz_2^2$ on the complex affine plane $\cc^2\subset\cp^2$. 
We denote the infinity line in $\cp^2$ by $\oc_{\infty}=\cp^2\setminus\cc^2$.   

\begin{definition}  The {\it symmetry} $\cc^2\to\cc^2$ with respect to a non-isotropic 
complex line $L\subset\cp^2$  is the unique non-trivial complex-isometric involution 
fixing the points of $L$. It extends to a projective transformation of the ambient plane $\cp^2$. 
For every $x\in L$ it acts on the space $\mcl_x=\cp^1$ of lines through $x$, and this action is called {\it symmetry at $x$}. 
If $L$ is an isotropic line through a finite point $x$, then a pair of  lines through $x$ is called symmetric with respect to $L$, if 
it is a limit of symmetric pairs of lines with respect to  non-isotropic lines converging to $L$. 
\end{definition}

\begin{lemma} \label{lim-refl}  Let $L$ be an isotropic line through a finite point $x$. A pair of lines $(L_1,L_2)$ through $x$ is 
symmetric with respect to $L$, if and only if some of them coincides with $L$.  
\end{lemma} 

\begin{convention} Everywhere below given an analytic curve $a\subset\cp^n$ and $A\in\ha$, we set $A'=\pi_a(A)$. By $T_Aa$ we denote the 
tangent line at $A'$ to the germ of curve $\pi_a:(\ha,A)\to(a,A')$. 
\end{convention}

\begin{definition} 
 Let $a_1,\dots,a_k\subset\cp^2$ be an analytic (algebraic) billiard, and let $\ha_1,\dots,\ha_k$ be the maximal normalizations of its 
 mirrors. The {\it completed $k$-periodic set} is the closure of the set of 
 those $k$-gons $A_1\dots A_k\in\hat a_1\times\dots\times\hat a_k$ for which 
 $A_1'\dots A_k'$  is a $k$-periodic billiard orbit.
 \end{definition}
  
 \begin{proposition} \label{comp-set} 
 The completed $k$-periodic  set $U$ is analytic 
 (algebraic). The billiard is $k$-reflective, if and only if the set $U$ has at least one two-dimensional irreducible component $U_0\subset U$ 
(which will be called the {\bf $k$-reflective component}).  For every point 
 $A_1\dots A_k\in U$ and every $j$ such that $A_{j-1}'\neq A_j'$ and $A_j'\neq A_{j+1}'$ the {\bf complex reflection law} holds: 
 
 - if the tangent line $l_j=T_{A_j}a_j$ is not isotropic, then the lines $A_{j-1}'A_j'$ and $A_j'A_{j+1}'$ are symmetric with respect to 
 $l_j$;
 
 - otherwise, if $l_j$ is isotropic (finite or infinite), then at least one of the lines $A_{j-1}'A_j'$ or $A_j'A_{j+1}'$ coincides with $l_j$.

If the billiard is $k$-reflective, then each projection $U_0\to\ha_j\times\ha_{j+1}$ is a submersion on an open dense subset (epimorphic, if 
the billiard is algebraic). 
\end{proposition}

\begin{definition} Let $a_1,\dots,a_k$ be a complex planar analytic (algebraic) billiard. 
A point $P\in\cp^2$ is {\it marked}, if it is either a cusp, or an isotropic tangency point of some mirror $a_j$. A point $P$ is {\it double}, 
if it is either a self-intersection of a mirror, or an intersection point of two distinct mirrors. 
\end{definition}

\begin{corollary} \label{ccoinc} Let $a_1,\dots, a_k$ be a $k$-reflective analytic billiard in $\cp^2$. Let $A_1\dots A_k\in\ha_1\times\dots\times\ha_k$ 
be a point of a $k$-reflective component, and let $A_j'=A_{j+1}'$ for some $j$. Then we have one of the following possibilities:

(i) $A'_j$ is either a marked, or a double point;

(ii) $a_1=\dots=a_k$, $A_1'=\dots=A_k'$; 

(iii) up to cyclic mirror renaming, 
there exists an $s<j$ such that $a_{s+1}=\dots=a_j$, $A_s'\neq A_{s+1}'=\dots= A_j'$, and the line $A_s'A_j'$ coincides with $T_{A_j}a_j$. 
\end{corollary}

\begin{proof} Everywhere below we consider that  the point $A_j'$ is neither marked, nor double: otherwise we have case (i). 
If $A_1'=\dots=A_k'$, then $a_1=\dots=a_k$, since otherwise the latter point, which 
coincides with $A_j'$, would be double, -- a contradiction. Thus, in this case we have (ii). 
Let now there exist an $s\in\{1,\dots,k\}$ such that $A_s'\neq A_j'$. Without loss of generality we 
consider that $s<j$ (after a possible cyclic mirror renaming), and we take the maximal $s$ as above. One has $a_{s+1}=\dots=a_j$, as in the above argument, 
and $A_{s+1}'=\dots=A_j'$. Let us show that $A_s'A_j'=T_{A_j}a_j$. By definition, 
the point $A_1\dots A_k$ is a limit of points $A_{1,n}\dots A_{k,n}$ corresponding to $k$-periodic billiard orbits, in particular, 
$A_{i,n}'\neq A_{i+1,n}'$ for all $i=1,\dots,k$. Thus, the distinct points $A_{s+1,n}'$ and 
$A_{s+2,n}'$ of the curve $a_j$ collide to the same limit $A_j'$, which is neither marked, nor double  point, 
while $A_{s,n}'$ and $A_{s+1,n}'$ don't collide in the limit.  Hence, $A_{s+1,n}'A_{s+2,n}'\to T_{A_j}a_j$. 
This together with the reflection law implies that  the limit line $A_s'A_{s+1}'=A_s'A_j'=\lim(A_{s,n}'A_{s+1,n}')$ coincides with $T_{A_j}a_j$. 
Thus, we have case (iii). This proves the corollary.
\end{proof}

\section{Algebraic billiards: proof of Theorem \ref{odd}}

As it is shown below, Theorem \ref{odd} is implied by the following proposition. 

\begin{proposition} \label{pisotrop} Let $a_1,\dots,a_k$ be a $k$-reflective planar algebraic billiard such that each mirror $a_j$ contains no 
isotropic point at infinity. Then it has at least one  finite $k$-periodic orbit with  an isotropic 
edge.  Moreover, the latter orbit can be realized by a point of a $k$-reflective component. 
\end{proposition}

\begin{proof} 
Let $U\subset\ha_1\times\dots\times\ha_k$ be a $k$-reflective component, see the Proposition \ref{comp-set}. 
Let $W_{12}\subset\ha_1\times\ha_2$ denote the Zariski closure of the set of those pairs of points 
$(A_1,A_2)$ whose projections $A_1'$ and $A_2'$ are distinct, finite 
and  for which the line $A_1'A_2'$ is an isotropic line through the isotropic point $I_1$ at infinity. We show that the pairs from a 
non-empty Zariski open subset in $W_{12}$ extend to orbits as in Proposition \ref{pisotrop}. This will prove the proposition. 
\medskip

{\bf Claim 1.} {\it The set $W_{12}$ is non-empty, and hence, it is an algebraic curve.}

\begin{proof} Suppose the contrary. Then each line through $I_1$ intersects the union $a_1\cup a_2$ in at most one finite point. 
This together with the assumption that $I_1\notin a_j$ implies that $a_1=a_2$ is a line. But in this case there would be no 
$k$-periodic orbits at all. Indeed, in a $k$-periodic orbit $A_1'\dots A_k'$ the line $A_1'A_2'$ should coincide with $a_1$, and hence, it cannot 
be transverse to $T_{A_1}a_1$, -- a contradiction to Definition \ref{deforb}. The contradiction to $k$-reflectivity thus obtained 
proves the claim.
 \end{proof}
 
%
%

Let $W\subset U$ denote the preimage in $U$ of the curve $W_{12}$ under the product projection to $\ha_1\times\ha_2$. The projection 
$W\to W_{12}$ is epimorphic, by Proposition \ref{comp-set}. 
For every $j=2,\dots,k+1$ let $W_j\subset W$ denote the set of points $A_1\dots A_k\in W$  such that for every $i\leq j$ the point $A_i'$ 
is finite, neither marked, nor double,   and $A_i'\neq A_{i-1}'$. By definition, 
one has $W_2\supset W_3\supset\dots\supset W_{k+1}$. We 
show  simultaneously by induction in $j$ that 

A) the subset $W_j\subset W$ is Zariski open and non-empty;

B) the product projection $W_j\to \ha_{j+1}$ is locally non-constant. 

The points of the set $W_{k+1}$ correspond to orbits as in Proposition \ref{pisotrop}. This will prove the proposition.

Induction base. Statement A) for $j=2$  follows from the above claim and the obvious fact that the points $A_1'$ and $A_2'$ vary along the 
curve $W_{12}$. Let us prove statement B). Suppose the contrary: there exists an open 
subset in $W$ of points  $A_1\dots A_k$ that are projected 
to one and the same point $Q=A_3\in \ha_3$. Then there exists an open set of  finite points $A_2'\in a_2$ 
such that the image of the isotropic line $A_2I_1$ under the symmetry with respect to the tangent line $T_{A_2}a_2$ passes through 
one and the same point $Q$. This follows by definition, Claim 1 and the epimorphicity of the projection $W\to W_{12}$.  
On the other hand, the above image should pass through the other isotropic point $I_2$, by Proposition \ref{is-refl}. Hence, $Q=I_2\in a_3$, -- 
a contradiction to the assumption that the mirrors 
$a_j$ contain no isotropic points at infinity. The induction base is proved. 

Induction step. Let the statements A) and B) be proved for all $j\leq r\leq k$. Let us prove them for $j=r+1$.  
For the proof of statement A) it suffices to show that the set of those points $A_1\dots A_k\in W_r$ for which the point  
$A_{r+1}'$ is finite, neither marked, nor double and distinct from $A_r'$ 
is Zariski open in $W_r$ and non-empty. Indeed, on a non-empty Zariski open subset  $\wt W\subset W_r$ the points $A_r'$ and $A_{r+1}'$ are 
 finite and neither marked, nor double,  
by statement B) for $j=r-1,r$ (the induction hypothesis). The line $A_{r-1}'A_r'$ is isotropic, being the image of an 
isotropic line $A_1'A_2'$ under a finite number of non-isotropic reflections. Its image under the reflection from the line $T_{A_r}a_r$ is 
the isotropic line $L=A_r'A_{r+1}'$ 
through $A_r'$ transverse to $A_{r-1}'A_r'$. One has $L\neq T_{A_{r+1}}a_{r+1}$ on the above subset $\wt W$, 
since the point $A_{r+1}'$ is not marked. Therefore, $A_r'\neq A_{r+1}'$ on the same subset, 
by Corollary \ref{ccoinc} and since $A_r\not\equiv const$ along $\wt W$ (the induction hypothesis: statement B) for $j=r$). This proves statement A). 
The proof of statement B) repeats the argument from the induction base. The induction step is over.  Statements A) and B) are proved. 
Proposition \ref{pisotrop} is proved.
\end{proof}

Let us now prove Theorem \ref{odd}. Suppose the contrary: there exists a $k$-reflective billiard $a_1,\dots,a_k$ with odd $k$, whose 
mirrors contain no isotropic points at infinity. Then it has a finite $k$-periodic orbit with at least one isotropic edge  
(Proposition \ref{pisotrop}). But then $k$ should be even by Corollary \ref{is-even}. The contradiction thus obtained proves Theorem \ref{odd}.

\section{Triangular orbits: proof of Theorem \ref{three}} 

We prove Theorem \ref{three} by contradiction. Suppose the contrary: there exists a 3-reflective analytic billiard  $a$, $b$, $c$ in $\cp^2$, 
let $U\subset\ha\times\hb\times\hc$ be its 3-reflective component. First we show in the next proposition  that 
the correspondence $\psi_b:(A,B)\mapsto (B,C)$ defined by the triangles $ABC\in U$ induces a bimeromorphic isomorphism 
$\ha\times \hb\to \hb\times\hc$. This implies (Corollary \ref{parab}) that each mirror is either a rational curve, or a parabolic Riemann surface. 
  Afterwards we deduce  that the mirrors are distinct 
(Proposition \ref{is-edge}) and there exists a one-dimensional family of triangles $ABC\in U$ with 
 isotropic edges $A'B'$. We then show that the existence of the latter triangle family 
 would contradict the complex reflection law satisfied by the points 
of the set $U$. The contradiction thus obtained will prove Theorem \ref{three}.


\begin{proposition} \label{bim} Let $a$, $b$, $c$, $U$ and $\psi_b$ be as above. 
The correspondence $\psi_b$ extends to a bimeromorphic\footnote{Recall that a {\it meromorphic mapping} $M\to N$ between complex 
manifolds is a mapping holomorphic on the complement of an analytic subset in $M$ such that the closure of its graph is an 
analytic subset in $M\times N$.}  isomorphism $\ha\times\hb\to\hb\times\hc$. 
\end{proposition}

\begin{proof}  It suffices to show that the mapping $\psi_b$ is meromorphic: the proof of the meromorphicity of its inverse is analogous. 
Consider the auxiliary mapping $Q_{ab}:\ha\times\hb\to\cp^2$ defined as follows. 
Take an arbitrary pair $(A,B)\in\ha\times\hb$ with $A'\neq B'$ and such that the line $A'B'$ is neither tangent to $a$ at $A'$, nor  tangent 
to $b$ at $B'$. Set $Q_{ab}(A,B)$ to be the point of intersection of two lines: the images of the line $A'B'$ under the symmetries with 
respect to the tangent lines $T_{A}a$ and $T_{B}b$. 
The mapping $Q_{ab}$ extends to a meromorphic 
mapping $\ha\times\hb\to\cp^2$, by the algebraicity of the reflection law. (Possible indeterminacies correspond to  isolated points 
where either $A'=B'$ is a double point, or one of the tangent lines $T_Aa$ or $T_Bb$ is isotropic and coincides with $A'B'$.)  
Note that $Q_{ab}(A,B)\in\pi_c(\hc)$ for every $(A,B)$ from the domain of the mapping $Q_{ab}$: given two vertices $A'\in a$ and $B'\in b$ 
of a triangular billard orbit, the third vertex is found as the intersection point of the above symmetric images of the line $A'B'$. This 
implies that the mapping $\psi_b$ extends to a meromorphic mapping $\ha\times\hb\to\hb\times\hc$ by the formula 
$\psi_b(A,B)=(B,\pi_c^{-1}\circ Q_{ab}(A,B))$. The proposition is proved.
\end{proof}

\begin{corollary} \label{epi} In Proposition \ref{bim} the projection $U\to\ha\times\hb$ is  bimeromorphic. The complement to its image 
is contained in the indeterminacy set for the mapping $Q_{ab}$, and hence, is at most discrete. 
\end{corollary}

\begin{corollary} \label{parab} Let $a$, $b$, $c$ be a 3-reflective analytic billiard in $\cp^2$. 
Then the maximal normalization of each its mirror is 
either parabolic (having universal cover $\cc$), or conformally equivalent to the Riemann sphere.
\end{corollary}

\begin{proof} A Riemann surface has one of the two above types, if and only if it admits a nontrivial holomorphic family of conformal 
automorphisms. Thus, it suffices to show that the maximal normalization of each mirror has a nontrivial holomorphic family of 
automorphisms, or equivalently, has a nontrivial holomorphic family of conformal isomorphisms onto a given Riemann surface. 
Fix a point $B\in\hb$ such that $B'$ is finite and not marked. For every $A\in\ha$ set 
$\phi_B(A)=\pi_c^{-1}\circ Q_{ab}(A,B)\in\hc$. This yields a family of conformal isomorphisms $\phi_B:\ha\to\hc$ depending holomorphically on $B\in\hb$, 
by bimeromorphicity (Proposition \ref{bim}). In particular, the Riemann surfaces $\ha$ and $\hc$ are conformally equivalent. 
Similarly, $S=\ha\simeq\hb\simeq\hc$. If the family $\phi_B$ is nontrivial (non-constant in $B$), then the Riemann surface $S$ is 
either parabolic, or the Riemann sphere, by the statement from the beginning of the proof. We claim that in the contrary case, 
 when $\phi_B$ is independent on $B$, one has $b\simeq\oc$. Indeed, let $\phi=\phi_B$ be 
 independent on $B$. Fix an arbitrary 
 $A\in\ha$ such that $A'$ is finite; set $C=\phi(A)$. Then {\it for every} $B\in\hb$ the lines $A'B'$ and $B'C'$ are symmetric with respect to the tangent line $T_{B}b$. Hence, $b$ is either a line, or a conic. Thus,  $b\simeq\oc$. This proves the corollary.
\end{proof}

\begin{proposition} \label{is-edge} Let $a$, $b$, $c$ be a 3-reflective analytic billiard in $\cp^2$. Then its mirrors are pairwise distinct: 
 one is not analytic extension of another.
 \end{proposition} 
 
 \begin{proof} Suppose the contrary, say, $a=b$. Let $U$ be a $k$-reflective component. Then $U$ contains an analytic 
 curve $\Gamma$ consisting of those triples $ABC$ for which $A'=B'$ (Corollary \ref{epi}). Let us fix its irreducible component and 
 denote $\Gamma$ the latter component. Let $A'\equiv B'\not\equiv C'$ on $\Gamma$. Then  
 $C'\in T_{A}a\cap c$ for every $ABC\in\Gamma$ (Corollary \ref{ccoinc}). This implies that $C\not\equiv const$ along the curve 
 $\Gamma$, and hence, it is neither marked, 
 nor double outside a countable subset in $\Gamma$. Thus, the curve $\Gamma$ contains triples $ABC$ such that $C'\neq A'=B'$ and $C'$ 
 is neither marked, nor double point. This 
 contradicts the second proposition in \cite[subsection 2.4]{alg}. In the case, when $A'\equiv B'\equiv C'$ on $\Gamma$, 
 we similarly get a contradiction to the same proposition. This proves Proposition \ref{is-edge}. 
 \end{proof} 
 
 \begin{proof} {\bf of Theorem \ref{three}.} Suppose the contrary: there exists a 3-reflective billiard $a$, $b$, $c$. Let $U$ be a 3-reflective 
 component.  Let us show that  there exists an analytic curve $\Gamma\subset U$ consisting of triples $ABC$ such that 
 $A'\neq B'$, the points $A'$ and $B'$ are finite  and  the line $A'B'$ is isotropic through the point $I_1$.  
 As it is shown below, this curve $\Gamma$ cannot exist.  
 Consider the projections $\ha,\hb\to\cp^1$: the compositions of the parametrizations $\pi_a$, $\pi_b$ with the projection 
 from the isotropic point $I_1$. Each of them is holomorphic 
 and takes all the values, except for at most two, since both maximal normalizations $\ha$, $\hb$ are either parabolic, or 
 conformally equivalent to $\oc$ (Corollary \ref{parab}) and by Picard's Theorem. 
 This together with Proposition \ref{is-edge}  implies that there exists a line 
 through $I_1$ that contains two distinct finite points $A'\in a$ and $B'\in b$. This together with Corollary \ref{epi} implies that 
 the above-defined set $\Gamma$ is non-empty and is an analytic curve. 
 
Note that both $A$ and $B$ are non-constant along the curve $\Gamma$, and for every $ABC\in\Gamma$ such that $A'$ 
and $B'$ are not marked points the lines $A'C'$ and $B'C'$ are isotropic lines through $I_2$. The latter follows from reflection law, 
Proposition \ref{is-refl} and the inclusion $I_1\in A'B'$. This implies that $C'\equiv I_2$ on $\Gamma$. 
Thus, the point $I_2$ is contained in (the maximal analytic extension of the curve) $c$ and by definition, 
the tangent line $T_{I_2}c$  to any branch of the curve $c$ through $I_2$  is isotropic. 
The reflection images $A'I_2$, $B'I_2$ of the line $A'B'$ with respect to the tangent lines $T_{A}a$ and $T_{B}b$ vary, as $ABC$ ranges 
along a component of the curve $\Gamma$, since $A'$ and $B'$ vary and the curves $a$, $b$ are not isotropic lines, see Fig.4. 
On the other hand, at least one of the lines $A'I_2$, $B'I_2$ should coincide with one and the same tangent line  $T_{I_2}c$, 
by Proposition \ref{comp-set} (reflection law), -- a contradiction. The proof of Theorem \ref{three} is complete.
\end{proof}

\begin{figure}[ht]
  \begin{center}
   \epsfig{file=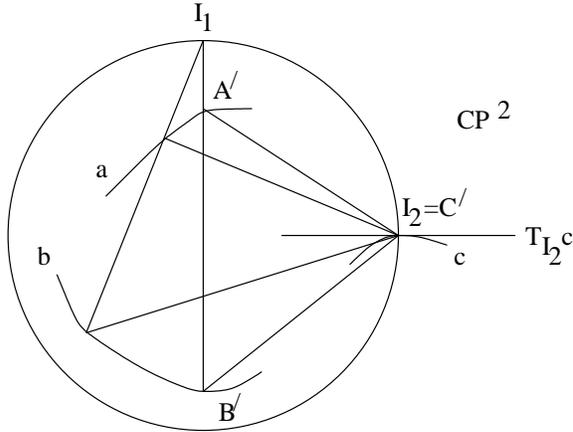}
    \caption{A family of triangular  orbits with isotropic edges $A'B'$}
    \label{fig:4}
  \end{center}
\end{figure}

\section{Acknowledgements} 
I am grateful to Yu.S.Ilyashenko, Yu.G.Kudryashov and A.Yu.Plakhov for attracting my attention to Ivrii's conjecture and invisibility 
and for helpful discussions.

\end{document}